\documentclass{article}
\usepackage{amsfonts}

\newtheorem{theorem}{Theorem}

\newtheorem{corollary}[theorem]{Corollary}

\newtheorem{remark}[theorem]{Remark}

\newenvironment{proof}[1][Proof]{\noindent\textbf{#1.} }{\ \rule{0.5em}{0.5em}}
\input{tcilatex}
\begin{document}

\title{Tornheim-like series, harmonic numbers and zeta values}
\author{Ilham A. Aliev \\
Department of Mathematics,\\
Akdeniz University, 07058 Antalya Turkey\\
ialiev@akdeniz.edu.tr \and Ayhan Dil* \\
Department of Mathematics,\\
Akdeniz University, 07058 Antalya Turkey\\
adil@akdeniz.edu.tr*}
\maketitle

\begin{abstract}
Explicit evaluations of the Tornheim-like double series in the form%
\[
\sum_{n,m=1}^{\infty }\frac{H_{n+m+s}}{nm\left( n+m+s\right) },\text{ }s\in 
\mathbb{N\cup }\left\{ 0\right\} 
\]%
and their extensions are given. Furthermore, series of the type 
\[
\sum_{m=1}^{\infty }\frac{2H_{2m+1}-H_{m}}{2m\left( 2m+1\right) } 
\]%
and some other Tornheim-like multiple series\ are evaluated in terms of the
zeta values.

\textbf{2010 Mathematics Subject Classification: }40A25\textbf{, }40B05,
11M06\textbf{.}

\textbf{Key words: }Tornheim series, harmonic numbers, Riemann Zeta values.
\end{abstract}

\newpage 

\section{Introduction}

Riemann zeta function is defined by%
\[
\zeta \left( s\right) =\sum_{k=1}^{\infty }\frac{1}{k^{s}}, 
\]%
where $s=\sigma +it$ and $\sigma >1$. For even positive integers, one has
the well-known relationship between zeta values and Bernoulli numbers:%
\begin{equation}
\zeta \left( 2n\right) =\left( -1\right) ^{n+1}\frac{\left( 2\pi \right)
^{2n}}{2\left( 2n\right) !}B_{2n}.  \label{zb}
\end{equation}%
Here $B_{0}=1,$ $B_{1}=-\frac{1}{2},$ $B_{2}=\frac{1}{6},$ $B_{4}=-\frac{1}{%
30},$ $B_{6}=\frac{1}{42},...$ and $B_{2n+1}=0$ for $n\geq 1$ (this result
first published by Euler in 1740).

For odd positive integers, no such simple expression as (\ref{zb}) is known.
Roger Ap\'{e}ry \cite{AR} proved the irrationality of $\zeta \left( 3\right) 
$ and after that $\zeta \left( 3\right) $ was named as Ap\'{e}ry's constant.
Rivoal \cite{TR}\ has shown that infinitely many of the numbers $\zeta
\left( 2n+1\right) $ must be irrational. Besides Zudilin \cite{WZ} has shown
that at least one of the numbers $\zeta \left( 5\right) $, $\zeta \left(
7\right) $, $\zeta \left( 9\right) $ and $\zeta \left( 11\right) $ is
irrational.

The $n$th harmonic number $H_{n}$ is the $n$th partial sum of the harmonic
series:%
\[
H_{n}:=\sum_{k=1}^{n}\frac{1}{k}. 
\]%
For a positive integer $n$ and an integer $m$ the $n$th generalized harmonic
number of order $m$ is defined by%
\[
H_{n}^{\left( m\right) }:=\sum_{k=1}^{n}\frac{1}{k^{m}}, 
\]%
which is the $n$-th partial sum of the Riemann zeta function $\zeta \left(
m\right) $.

Tornheim double series \cite{Tor} (or so called Witten's zeta function \cite%
{DZ}) is defined by%
\begin{equation}
S\left( a,b,c\right) :=\sum_{m,n=1}^{\infty }\frac{1}{m^{a}n^{b}\left(
m+n\right) ^{c}}.  \label{TDS}
\end{equation}%
This series has attracted considerable attention in recent years and been
proved to be powerful tool to find numerous interesting relations between
various zeta values (\cite{BAB, BA, B, BB, BD, E1, E2, K, XZ}). Boyadzhiev 
\cite{KB1, KB2} described a simple method to evaluate multiple series of the
form (\ref{TDS}) in terms of zeta values.

It is well known that there exist deep connections between Tornheim type
series, harmonic numbers and zeta values. As a simple and nice example the
following equation can be given (see \cite{B, BB, E1, E2, N}):

\begin{equation}
\sum_{n,m=1}^{\infty }\frac{1}{nm\left( n+m\right) }=\sum_{m=1}^{\infty }%
\frac{H_{m}}{m^{2}}=2\zeta \left( 3\right) .  \label{hz}
\end{equation}

Kuba \cite{K} considered the following general sum:%
\[
V=\sum_{j,k=1}^{\infty }\frac{H_{j+k}^{\left( u\right) }}{j^{r}k^{s}\left(
j+k\right) ^{t}}. 
\]%
This sum includes the Tornheim's double series (\ref{TDS})\ as special case,
Kuba \cite{K} proved that whenever $w=r+s+t+u$ is even, for $r,s,t,w\in 
\mathbb{N}$ , the series $V$ can be explicitly evaluated in terms of zeta
functions.

On the other hand, Xu and Li \cite{XZ} used the Tornheim type series
computations for evaluation of non-linear Euler sums. Among other results
they obtained%
\begin{equation}
\sum_{m=1}^{\infty }\frac{H_{m+k}}{m\left( m+k\right) }=\frac{%
H_{k}^{2}+H_{k}^{\left( 2\right) }}{k},k\in \mathbb{N}=\left\{ 1,2,3,\ldots
\right\} .  \label{XL}
\end{equation}%
From (\ref{hz}) and (\ref{XL}) it is easy to see that the value of series%
\[
a\left( k\right) =\sum_{m=1}^{\infty }\frac{H_{m+k}}{m\left( m+k\right) }%
,k\in \mathbb{N\cup }\left\{ 0\right\} 
\]%
is irrational for $k=0$ and rational for every $k\in \mathbb{N}$. Hence the
following questions arise naturally: for the integer $s\in \mathbb{N\cup }%
\left\{ 0\right\} $, the values of the double series%
\[
\sum_{n,m=1}^{\infty }\frac{H_{n+m+s}}{nm\left( n+m+s\right) }, 
\]%
and more generally, multiple series%
\[
A_{n}\left( s\right) =\sum_{k_{1}=1}^{\infty }\ldots
\sum_{k_{n-1}=1}^{\infty }\frac{H_{k_{1}+\cdots +k_{n-1}+s}}{k_{1}\cdots
k_{n-1}\left( k_{1}+\cdots +k_{n-1}+s\right) } 
\]%
are rational or irrational? In this work we answer these questions.
Moreover, we give explicit evaluation formulas for some Tornheim-like series
via zeta values.

\section{Formulations and proofs of the main results}

\begin{theorem}
\label{it}Consider the double series%
\[
A\left( s\right) =\sum_{n,m=1}^{\infty }\frac{H_{n+m+s}}{nm\left(
n+m+s\right) },\text{ \ \ }s\in \mathbb{N\cup }\left\{ 0\right\} . 
\]%
For any $s\in \mathbb{N}$ the value of $A\left( s\right) $ is rational but $%
A\left( 0\right) $ is irrational. More precisely,%
\[
A\left( s\right) =\left\{ 
\begin{array}{cc}
6\zeta \left( 4\right) & \text{if }s=0 \\ 
6\sum_{j=0}^{s-1}\left( -1\right) ^{j}\binom{s-1}{j}\frac{1}{\left(
j+1\right) ^{4}} & \text{if }s\geq 1%
\end{array}%
\right. . 
\]
\end{theorem}

\begin{proof}
By telescoping series formula we have%
\begin{eqnarray*}
1+\frac{1}{2}+\frac{1}{3}+\cdots +\frac{1}{n+m+s} &=&\sum_{k=1}^{\infty
}\left( \frac{1}{k}-\frac{1}{k+n+m+s}\right) \\
&=&\left( n+m+s\right) \sum_{k=1}^{\infty }\frac{1}{k\left( k+n+m+s\right) }.
\end{eqnarray*}%
It then follows that%
\begin{eqnarray*}
A\left( s\right) &=&\sum_{n,m=1}^{\infty }\frac{1}{nm\left( n+m+s\right) }%
\left( 1+\frac{1}{2}+\frac{1}{3}+\cdots +\frac{1}{n+m+s}\right) \\
&=&\sum_{n,m,k=1}^{\infty }\frac{1}{nmk\left( n+m+k+s\right) } \\
&=&\sum_{n,m,k=1}^{\infty }\left( \int_{0}^{1}x^{n-1}dx\right) \left(
\int_{0}^{1}y^{m-1}dy\right) \left( \int_{0}^{1}z^{k-1}dz\right) \left(
\int_{0}^{1}t^{n+m+k+s-1}dt\right) \\
&=&\int_{0}^{1}t^{s+2}\left[ \int_{0}^{1}\left( \sum_{n=1}^{\infty }\left(
xt\right) ^{n-1}\right) dx\int_{0}^{1}\left( \sum_{m=1}^{\infty }\left(
yt\right) ^{m-1}\right) dy\int_{0}^{1}\left( \sum_{k=1}^{\infty }\left(
zt\right) ^{k-1}\right) dz\right] dt \\
&=&\int_{0}^{1}t^{s+2}\left[ \int_{0}^{1}\frac{1}{1-xt}dx\int_{0}^{1}\frac{1%
}{1-yt}dy\int_{0}^{1}\frac{1}{1-zt}dz\right] dt.
\end{eqnarray*}%
Since%
\[
\int_{0}^{1}\frac{1}{1-ut}du=-\frac{1}{t}\ln \left( 1-t\right) , 
\]%
we have%
\begin{equation}
A\left( s\right) =-\int_{0}^{1}t^{s-1}\ln ^{3}\left( 1-t\right)
dt=-\int_{0}^{1}\left( 1-t\right) ^{s-1}\ln ^{3}tdt.  \label{u1}
\end{equation}%
Setting $s=0$, it follows that%
\begin{eqnarray*}
A\left( 0\right) &=&-\int_{0}^{1}\frac{1}{1-t}\ln
^{3}tdt=-\sum_{j=0}^{\infty }\int_{0}^{1}t^{j}\ln ^{3}tdt \\
&=&-\sum_{j=0}^{\infty }\left( -\frac{6}{\left( j+1\right) ^{4}}\right)
=6\zeta \left( 4\right) =\frac{\pi ^{4}}{16}.
\end{eqnarray*}%
On the other hand, if $s\geq 1$, then utilizing the formulas%
\[
\left( 1-t\right) ^{s-1}=\sum_{j=0}^{s-1}\left( -1\right) ^{j}\binom{s-1}{j}%
t^{j} 
\]%
and%
\[
\int_{0}^{1}t^{j}\ln ^{3}tdt=-\frac{3!}{\left( j+1\right) ^{4}}, 
\]%
(\ref{u1}) can computed explicitly as%
\begin{eqnarray*}
A\left( s\right) &=&-\int_{0}^{1}\left( 1-t\right) ^{s-1}\ln ^{3}tdt \\
&=&3!\sum_{j=0}^{s-1}\left( -1\right) ^{j}\binom{s-1}{j}\frac{1}{\left(
j+1\right) ^{4}}.
\end{eqnarray*}%
This proves the stated result.
\end{proof}

In the same way as in Theorem \ref{it}, by making use of the formulas,%
\[
\left( 1-t\right) ^{s-1}=\sum_{j=0}^{s-1}\left( -1\right) ^{j}\binom{s-1}{j}%
t^{j}\text{ and }\int_{0}^{1}t^{j}\ln ^{k}tdt=\left( -1\right) ^{k}\frac{k!}{%
\left( j+1\right) ^{k+1}}, 
\]%
one can prove the following more general result.

\begin{theorem}
\label{G}Denote%
\[
A_{n}\left( s\right) =\sum_{k_{1}=1}^{\infty }\ldots
\sum_{k_{n-1}=1}^{\infty }\frac{H_{k_{1}+\cdots +k_{n-1}+s}}{k_{1}\cdots
k_{n-1}\left( k_{1}+\cdots +k_{n-1}+s\right) },\text{ \ \ }s\in \mathbb{%
N\cup }\left\{ 0\right\} . 
\]%
Then%
\begin{equation}
A_{n}\left( s\right) =\left\{ 
\begin{array}{cc}
n!\zeta \left( n+1\right) & \text{if }s=0, \\ 
n!\sum_{j=0}^{s-1}\left( -1\right) ^{j}\binom{s-1}{j}\frac{1}{\left(
j+1\right) ^{n+1}} & \text{if }s\geq 1.%
\end{array}%
\right.  \label{g}
\end{equation}
\end{theorem}

Two special cases of the theorem are as follows:%
\[
A_{2}\left( s\right) =\sum_{k=1}^{\infty }\frac{H_{k+s}}{k\left( k+s\right) }%
=\left\{ 
\begin{array}{cc}
2!\zeta \left( 3\right) & \text{if }s=0, \\ 
2!\sum_{j=0}^{s-1}\left( -1\right) ^{j}\binom{s-1}{j}\frac{1}{\left(
j+1\right) ^{3}} & \text{if }s\geq 1,%
\end{array}%
\right. 
\]%
and%
\[
A_{4}\left( s\right) =\sum_{k,m,n=1}^{\infty }\frac{H_{k+m+n+s}}{kmn\left(
k+m+n+s\right) }=\left\{ 
\begin{array}{cc}
4!\zeta \left( 5\right) & \text{if }s=0, \\ 
4!\sum_{j=0}^{s-1}\left( -1\right) ^{j}\binom{s-1}{j}\frac{1}{\left(
j+1\right) ^{5}} & \text{if }s\geq 1.%
\end{array}%
\right. 
\]

\begin{remark}
It can be easily seen from (\ref{g}) that the expression $A_{n}\left(
s\right) $ is a rational number for all $s\geq 1$, but $A_{2}\left( 0\right)
=2!\zeta \left( 3\right) $ is irrational. If $n\geq 4$ and even, it is not
known whether the numbers $A_{n}\left( 0\right) =n!\zeta \left( n+1\right) $
are irrational or not. On the other hand, for any odd $n\in \mathbb{N}$ we
have $A_{n}\left( 0\right) =n!\zeta \left( n+1\right) =r_{n}\pi ^{n+1}$ (see
(\ref{zb})) is also irrational because of $r_{n}$ is rational and $\pi
^{n+1} $ is irrational. Notice that the irrationality of $\pi ^{n}$ is a
consequence of the transcendentality of $\pi $. Indeed, if $\pi ^{n}$ is
rational, say $\pi ^{n}=\frac{p}{q}$ where $p$ and $q$ are integers, then $%
\pi $ is a solution of the equation $qx^{n}-p=0$ and therefore $\pi $ must
be an algebraic number, which is false. More generally, if $\alpha $\ is a
transcendental number and $r=\frac{p}{q}$ is a rational number, then $\alpha
^{r}$ becomes irrational number.
\end{remark}

\begin{corollary}
For any $k\in \mathbb{N}$ we have%
\[
\sum_{j=0}^{k-1}\left( -1\right) ^{j}\binom{k-1}{j}\frac{1}{\left(
j+1\right) ^{3}}=\frac{H_{k}^{2}+H_{k}^{\left( 2\right) }}{2k}. 
\]
\end{corollary}

\begin{proof}
By applying Theorem \ref{G} in the case of $n=2$ and considering (\ref{XL})
we arrive at the stated result.
\end{proof}

The next theorem gives a new relationship between harmonic numbers and $%
\zeta \left( 2\right) $.

\begin{theorem}
Let%
\[
H_{m}=\sum_{k=1}^{m}\frac{1}{k}\text{ and }O_{m}=\sum_{k=1}^{m}\frac{1}{2k-1}%
. 
\]%
Then the formulas%
\begin{equation}
\sum_{m=1}^{\infty }\frac{2H_{2m+1}-H_{m}}{2m\left( 2m+1\right) }=2\left(
2-\ln 2\right) -\zeta \left( 2\right)  \label{ln}
\end{equation}%
and%
\begin{equation}
\sum_{m=1}^{\infty }\frac{O_{m}}{2m\left( 2m+1\right) }=\frac{1}{4}\zeta
\left( 2\right)  \label{on}
\end{equation}%
are valid.
\end{theorem}

\begin{proof}
Denote%
\[
A=\sum_{m,n=0}^{\infty }\frac{1}{\left( 2m+1\right) \left( 2n+1\right)
\left( 2m+2n+1\right) } 
\]%
and%
\[
B=\sum_{m,n=1}^{\infty }\frac{1}{\left( 2m+1\right) \left( 2n+1\right)
\left( 2m+2n+1\right) }. 
\]%
From the equation%
\[
\sum_{k=0}^{\infty }\frac{1}{\left( 2k+1\right) ^{2}}=\frac{\pi ^{2}}{8}, 
\]%
we have%
\begin{equation}
A=\frac{\pi ^{2}}{4}-1+B.  \label{2}
\end{equation}%
Further, by the telescoping series formula, we find%
\begin{eqnarray*}
B &=&\sum_{m=1}^{\infty }\frac{1}{2m+1}\frac{1}{2m}\sum_{n=1}^{\infty
}\left( \frac{1}{2n+1}-\frac{1}{2n+1+2m}\right) \\
&=&\sum_{m=1}^{\infty }\frac{1}{2m\left( 2m+1\right) }\left( \frac{1}{3}+%
\frac{1}{5}+\cdots +\frac{1}{2m+1}\right) \\
&=&\sum_{m=1}^{\infty }\frac{O_{m}-\frac{2m}{2m+1}}{2m\left( 2m+1\right) } \\
&=&\sum_{m=1}^{\infty }\frac{O_{m}}{2m\left( 2m+1\right) }%
-\sum_{m=1}^{\infty }\frac{1}{\left( 2m+1\right) ^{2}} \\
&=&\sum_{m=1}^{\infty }\frac{O_{m}}{2m\left( 2m+1\right) }-\frac{\pi ^{2}}{8}%
+1
\end{eqnarray*}%
Hence we obtain that%
\begin{equation}
B=\sum_{m=1}^{\infty }\frac{O_{m}}{2m\left( 2m+1\right) }+1-\frac{3}{4}\zeta
\left( 2\right) .  \label{s1}
\end{equation}%
On the other hand, it is clear that the expression $B$ can also be written as%
\begin{equation}
B=\sum_{m=1}^{\infty }\frac{H_{2m+1}-1-\frac{1}{2}H_{m}}{\left( 2m+1\right)
2m}.  \label{3}
\end{equation}%
Now let us evaluate $A$.%
\begin{eqnarray*}
A &=&\sum_{m,n=0}^{\infty }\left( \int_{0}^{1}x^{2m}dx\right) \left(
\int_{0}^{1}y^{2n}dy\right) \left( \int_{0}^{1}t^{2m+2n}dt\right) \\
&=&\int_{0}^{1}\left( \int_{0}^{1}\sum_{m=0}^{\infty }\left( xt\right)
^{2m}dx\int_{0}^{1}\sum_{n=0}^{\infty }\left( yt\right) ^{2n}dy\right) dt \\
&=&\int_{0}^{1}\left( \int_{0}^{1}\frac{1}{1-\left( xt\right) ^{2}}%
dx\int_{0}^{1}\frac{1}{1-\left( yt\right) ^{2}}dy\right) dt \\
&=&\frac{1}{4}\int_{0}^{1}\frac{1}{t^{2}}\ln ^{2}\left( \frac{1+t}{1-t}%
\right) dt.
\end{eqnarray*}%
The substitution $\frac{1+t}{1-t}=u$ immediately leads to the following
equality:%
\[
A=\frac{1}{2}\int_{1}^{\infty }\frac{1}{\left( 1-u\right) ^{2}}\ln ^{2}udu. 
\]%
Integration by parts gives%
\begin{eqnarray}
A &=&\int_{1}^{\infty }\frac{1}{u\left( u-1\right) }\ln udu=\int_{1}^{\infty
}\frac{1}{1-\frac{1}{u}}\frac{\ln u}{u^{2}}du  \nonumber \\
&=&\sum_{k=0}^{\infty }\int_{1}^{\infty }u^{-k-2}\ln udu=\sum_{k=0}^{\infty }%
\frac{1}{\left( k+1\right) ^{2}}=\zeta \left( 2\right) =\frac{\pi ^{2}}{6}.
\label{4}
\end{eqnarray}%
We use (\ref{2}), (\ref{3}) and (\ref{4}) to conclude that%
\[
\sum_{m=1}^{\infty }\frac{H_{2m+1}-1-\frac{1}{2}H_{m}}{2m\left( 2m+1\right) }%
=1-\frac{\pi ^{2}}{4}+\frac{\pi ^{2}}{6}=1-\frac{\pi ^{2}}{12}, 
\]%
from which we obtain%
\[
\sum_{m=1}^{\infty }\frac{H_{2m+1}-\frac{1}{2}H_{m}}{2m\left( 2m+1\right) }%
=1-\frac{\pi ^{2}}{12}+\sum_{m=1}^{\infty }\frac{1}{2m\left( 2m+1\right) }. 
\]%
From the formulas%
\[
\frac{1}{1.2}+\frac{1}{2.3}+\frac{1}{3.4}+\cdots =1 
\]%
and%
\[
\frac{1}{1.2}+\frac{1}{3.4}+\frac{1}{5.6}+\cdots =\ln 2, 
\]%
it follows that%
\[
\sum_{m=1}^{\infty }\frac{1}{2m\left( 2m+1\right) }=1-\ln 2 
\]%
and therefore%
\begin{eqnarray*}
\sum_{m=1}^{\infty }\frac{H_{2m+1}-\frac{1}{2}H_{m}}{2m\left( 2m+1\right) }
&=&2\left( 1-\frac{\pi ^{2}}{12}+1-\ln 2\right) \\
&=&2\left( 2-\ln 2\right) -\zeta \left( 2\right) .
\end{eqnarray*}%
Similarly, from (\ref{2}), (\ref{s1}) and (\ref{4}) we have%
\[
\frac{\pi ^{2}}{6}=\frac{\pi ^{2}}{4}-1+\sum_{m=1}^{\infty }\frac{O_{m}}{%
2m\left( 2m+1\right) }-\frac{3}{4}\zeta \left( 2\right) +1 
\]%
and as a result%
\[
\sum_{m=1}^{\infty }\frac{O_{m}}{2m\left( 2m+1\right) }=\frac{1}{4}\zeta
\left( 2\right) . 
\]
\end{proof}

In the following theorem we give some interesting relationships between the
Tornheim-like series and the zeta values $\zeta \left( 2\right) $ and $\zeta
\left( 3\right) $.

\begin{theorem}
We have the following series evaluations:%
\[
\text{(}a\text{)\textbf{\ \ \ \ \ \ \ \ }}\sum_{m,n=0}^{\infty }\frac{1}{%
\left( m+\frac{1}{2}\right) \left( n+\frac{1}{2}\right) \left( m+n+\frac{1}{2%
}\right) \left( m+n+1\right) }=16\zeta \left( 2\right) -14\zeta \left(
3\right) . 
\]%
\[
\text{(}b\text{)\textbf{\ \ \ \ \ \ \ \ \ }}\sum_{m,n=0}^{\infty }\frac{1}{%
\left( m+\frac{1}{2}\right) \left( n+\frac{1}{2}\right) \left( m+n+1\right)
\left( m+n+\frac{3}{2}\right) }=14\zeta \left( 3\right) -8\zeta \left(
2\right) . 
\]%
\[
\text{(}c\text{)}\sum_{m,n=0}^{\infty }\frac{1}{\left( m+\frac{1}{2}\right)
\left( n+\frac{1}{2}\right) \left( m+n+\frac{1}{2}\right) \left(
m+n+1\right) \left( m+n+\frac{3}{2}\right) }=24\zeta \left( 2\right)
-28\zeta \left( 3\right) . 
\]
\end{theorem}

\begin{proof}
($a$) We have%
\begin{eqnarray*}
&&\sum_{m,n=0}^{\infty }\frac{1}{\left( 2m+1\right) \left( 2n+1\right)
\left( 2m+2n+2\right) } \\
&=&\sum_{m,n=0}^{\infty }\left( \int_{0}^{1}x^{2m}dx\right) \left(
\int_{0}^{1}y^{2n}dy\right) \left( \int_{0}^{1}t^{2m+2n+1}dt\right) \\
&=&\int_{0}^{1}t\left( \int_{0}^{1}\sum_{m=0}^{\infty }\left( xt\right)
^{2m}dx\int_{0}^{1}\sum_{n=0}^{\infty }\left( yt\right) ^{2n}dy\right) tdt \\
&=&\frac{1}{4}\int_{0}^{1}\frac{1}{t}\ln ^{2}\left( \frac{1+t}{1-t}\right)
dt.
\end{eqnarray*}%
Here the substitution $\frac{1+t}{1-t}=u$ leads to the following equality:%
\begin{eqnarray*}
\frac{1}{4}\int_{0}^{1}\frac{1}{t}\ln ^{2}\left( \frac{1+t}{1-t}\right) dt
&=&\frac{1}{2}\int_{1}^{\infty }\frac{1}{\left( u^{2}-1\right) }\ln ^{2}udu=%
\frac{1}{2}\int_{1}^{\infty }\frac{1}{u^{2}}\frac{1}{\left( 1-\frac{1}{u^{2}}%
\right) }\ln ^{2}udu \\
&=&\frac{1}{2}\sum_{k=0}^{\infty }\int_{1}^{\infty }u^{-2k-2}\ln udu.
\end{eqnarray*}%
After integration by parts we get%
\[
\frac{1}{2}\sum_{k=0}^{\infty }\int_{1}^{\infty }u^{-2k-2}\ln
udu=\sum_{k=0}^{\infty }\frac{1}{\left( 2k+1\right) ^{3}}=\frac{7}{8}\zeta
\left( 3\right) . 
\]%
Hence%
\begin{equation}
\sum_{m,n=0}^{\infty }\frac{1}{\left( 2m+1\right) \left( 2n+1\right) \left(
2m+2n+2\right) }=\frac{7}{8}\zeta \left( 3\right) .  \label{6}
\end{equation}%
On the other hand, according to the formula (\ref{4}),%
\begin{equation}
A=\sum_{m,n=0}^{\infty }\frac{1}{\left( 2m+1\right) \left( 2n+1\right)
\left( 2m+2n+1\right) }=\zeta \left( 2\right) .  \label{7}
\end{equation}%
Now, from (\ref{6}) and (\ref{7}) it follows that%
\begin{eqnarray*}
\zeta \left( 2\right) -\frac{7}{8}\zeta \left( 3\right)
&=&\sum_{m,n=0}^{\infty }\frac{1}{\left( 2m+1\right) \left( 2n+1\right)
\left( 2m+2n+1\right) } \\
&&-\sum_{m,n=0}^{\infty }\frac{1}{\left( 2m+1\right) \left( 2n+1\right)
\left( 2m+2n+2\right) } \\
&=&\sum_{m,n=0}^{\infty }\frac{1}{\left( 2m+1\right) \left( 2n+1\right)
\left( 2m+2n+1\right) \left( 2m+2n+2\right) }
\end{eqnarray*}%
and this proves ($a$).

($b$) By the same method in the proof of ($a$), we have%
\begin{eqnarray*}
&&\sum_{m,n=0}^{\infty }\frac{1}{\left( 2m+1\right) \left( 2n+1\right)
\left( 2m+2n+3\right) } \\
&=&\sum_{m,n=0}^{\infty }\left( \int_{0}^{1}x^{2m}dx\right) \left(
\int_{0}^{1}y^{2n}dy\right) \left( \int_{0}^{1}t^{2m+2n+2}dt\right) \\
&=&\int_{0}^{1}t^{2}\left( \int_{0}^{1}\frac{1}{1-\left( xt\right) ^{2}}%
dx\int_{0}^{1}\frac{1}{1-\left( yt\right) ^{2}}dy\right) dt \\
&=&\frac{1}{4}\int_{0}^{1}\ln ^{2}\left( \frac{1+t}{1-t}\right) dt \\
&=&\frac{1}{2}\int_{1}^{\infty }\frac{1}{\left( u+1\right) ^{2}}\ln ^{2}udu=-%
\frac{1}{2}\int_{1}^{\infty }\ln ^{2}ud\left( \frac{1}{u+1}\right) \\
&=&\int_{1}^{\infty }\frac{1}{u\left( u+1\right) }\ln udu=\sum_{k=2}^{\infty
}\left( -1\right) ^{k}\int_{1}^{\infty }u^{-k}\ln udu \\
&=&\sum_{k=2}^{\infty }\left( -1\right) ^{k}\frac{1}{\left( k-1\right) ^{2}}=%
\frac{\pi ^{2}}{12}.
\end{eqnarray*}%
Thus%
\begin{equation}
\sum_{m,n=0}^{\infty }\frac{1}{\left( 2m+1\right) \left( 2n+1\right) \left(
2m+2n+3\right) }=\frac{\pi ^{2}}{12}=\frac{1}{2}\zeta \left( 2\right) .
\label{b}
\end{equation}%
Now, from (\ref{6}) and (\ref{b}) we have%
\[
\frac{7}{8}\zeta \left( 3\right) -\frac{1}{2}\zeta \left( 2\right)
=\sum_{m,n=0}^{\infty }\frac{1}{\left( 2m+1\right) \left( 2n+1\right) \left(
2m+2n+2\right) \left( 2m+2n+3\right) } 
\]%
and this proves ($b$).

Finally, formula ($c$) can be obtained by substracting the formula ($b$)
from the formula ($a$).
\end{proof}

\begin{remark}
With the method used in this theorem, series of similar types containing
different combinations in the denominator, can be evaluated.
\end{remark}

\end{document}